\documentclass[english]{amsart}	

\usepackage{amsmath}

\usepackage{amssymb}

\usepackage{amsthm}

\usepackage{epsfig}
\usepackage{epstopdf}
\usepackage[T1]{fontenc}
\usepackage[latin9]{inputenc}
\usepackage{geometry}
\usepackage{verbatim}
\usepackage{float}
\usepackage{units}
\usepackage{textcomp}
\usepackage{mathrsfs}
\usepackage{amsthm}
\usepackage{amsmath}
\usepackage{amssymb}
\usepackage{graphicx}

\usepackage{pdflscape}

\usepackage{color}

\usepackage{babel}

\newtheorem{theorem}{Theorem}

\theoremstyle{plain}
\newtheorem{lemma}{Lemma}
\newtheorem{proposition}{Proposition}

\theoremstyle{definition}

\theoremstyle{remark}

\begin{document}

\title{A New Direct Proof of the Central Limit Theorem}

\date{24 July 2015}

\author{Vladimir \mbox{Dobri\ensuremath{\acute{\text{c}}}} }
\address{Department of Mathematics, Lehigh University, 14 East Packer Avenue, 
Bethlehem, PA 18015}
\email{vd00@lehigh.edu}

\author{Patricia Garmirian}
\address{Department of Mathematics, Tufts University, Boston Avenue, 
Medford, MA 02155}
\email{Patricia.Garmirian@tufts.edu (corresponding author)}

\thanks{Garmirian dedicates this paper to the memory of her dearly departed mentor, Vladimir Dobri\'c, and acknowledges the helpful comments of Lee Stanley, Rob Neel, and Daniel Conus.}



\begin{abstract}
We prove the Central Limit Theorem (CLT) from the definition of weak convergence using the Haar wavelet basis, calculus, and elementary probability. The use of the Haar basis pinpoints the role of $L^{2}([0,1])$ in the CLT as well as the assumption of finite variance. We estimate the rate of convergence and prove strong convergence off the tails.
\end{abstract}

\maketitle

\makeatletter

\section{Introduction}
The Central Limit Theorem (CLT) is one of the most fundamental theorems of probability theory. The theorem states that standardized sums of i.i.d. random variables having finite variance converge weakly to the standard normal distribution. As early as the 1770s, mathematicians were searching for the ``central limit," trying to establish the correct conditions for convergence in distribution and the formula for the limiting distribution. The connection between convergence in distribution and characteristic functions was established in the 1920s by L\'evy. 

The connection between convergence in distribution and weak convergence was established in the late 1940s (see [6]).  For a measurable space $(S,\mathcal{B}(S))$, where $S$ is a Polish space, a sequence of measures $\mu_{n}$ converges weakly to $\mu$ provided that for each bounded, continuous function $f:S\rightarrow{\mathbb{R}}$,
\[
\lim_{n\rightarrow{\infty}}\int_{S}^{}f(x)\, d\mu_{n}(x)=\int_{S}^{}f(x)\, d\mu(x).
\]
The advantage of using this definition is that it may serve as a stepping-stone to extending the CLT to random variables having values in more general spaces.

A stronger type of convergence for measures than ``weak" convergence is ``strong" convergence. For a measurable space $(S,\mathcal{B}(S))$, a sequence of measures $\mu_{n}$ converges strongly to $\mu$ provided that for each set $A\in{\mathcal{B}(S)}$,
\[
\lim_{n\rightarrow{\infty}}\mu_{n}(A)=\mu(A).
\]
We show that the type of convergence in the CLT is in fact strong off the tails.
 
In 1935, both Feller, [5],  and  L\'evy (independently), [8], [9], proved the Central Limit Theorem (CLT)  using characteristic functions.  
As the CLT is a fundamental theorem in probability theory, since the time of Donsker many experts have believed that there should be a direct proof (see [6] for example).  Also, there is obvious interest in determining rates and constants of convergence and these
characteristic function proofs gave no information about these issues.

Since 1935, there have been a number of more elementary or direct proofs of the CLT which do not use characteristic functions,
e.g., [1], [2], [3], [4], [10] and [11].   The last two prove the CLT directly from the definition of
weak convergence and [2] and (taken together) [1], [3], [4], do give a rate of convergence of $n^{-1/2}$ as well as a constant of convergence.  However, all of these proofs involve a hypothesis stronger than the optimal hypothesis of Feller and L\'evy, namely finite variance.
The hypothesis of [1], [2], [3], [4] and [11] is finite third moment, while the hypothesis of [10] is continuous second derivative
of the function $f$ in the definition of weak convergence.

We present a proof, directly from the definition of weak convergence, avoiding characteristic functions, which is optimal in terms of hypothesis (as in [5], [8], [9]).  We elaborate on the next complex of notions in what follows, but for now, it will suffice to say 
that the proof involves multinomial approximations to the initial sum of random variables, and identifies ``tails'' of both these
multinomials and the Gaussian.  We obtain strong convergence off these tails, i.e. the sum of the absolute values of the differences between the multinomial and Gaussian probabilities converges to zero.
Our proof also provides estimates of the rate and constant of convergence off these tails; the former is comparable to the rates of [2] and ([1], [3], [4]).

The translation into the language of random variables of the above definition of weak convergence is established by letting $\mu_{n}=P\circ{X_{n}^{-1}}$. A sequence of random variables $(X_{n})$ converges  weakly to a random variable $X$ if for each bounded, continuous function $f:\mathbb{R}\rightarrow{\mathbb{R}}$, 
\[
\lim_{n\rightarrow{\infty}}E(f(X_{n})) = E(f(X)),
\]
and it is this translation that we use in what follows. A sequence of random variables $X_{n}$ on a measurable space $(S,\mathcal{B}(S))$ converges strongly to a random variable $X$ if for each $A\in{\mathcal{B}(S)}$,
\[
\lim_{n\rightarrow{\infty}}P(X_{n}\in{A})=P(X\in{A}).
\] 
In the case where $X_{n}$ and $X$ are discrete with range set $J$, by the triangle inequality, it is sufficient to show that
\[
\lim_{n\rightarrow{\infty}}\sum_{j\in{J}}^{}\left|P(X_{n}=j)-P(X=j)\right|=0.
\] 
In fact, approximating the initial sum of random variables and the Gaussian by discrete versions, the preceding statement holds for the sequences in the CLT off the tails. 

Our proof employs the expansion of random variables on [0,1] (equipped with Lebesgue measure, on Borel sets) with respect to the Haar basis. The Haar basis is the simplest orthonormal system for $L^{2}([0,1])$. By considering random variables in $L^{2}([0,1])$, this proof is consistent with the assumption of finite variance. Therefore, the Haar basis is a natural tool for proving the CLT. 

The proof proceeds as follows: 
Given an i.i.d. sequence of random variables on a probability space, we construct an i.i.d. sequence on [0,1] with the Borel sigma algebra and Lebesgue measure having the same sequence of distributions. As the new sequence of random variables is defined on [0,1] and also has finite variance, we then expand this sequence with respect to the Haar basis. 

We then reduce the problem of showing weak convergence of this new sequence of random variables to the case where the Haar expansions are truncated to have only $M$ terms, for some finite $M$ which will be chosen to accomplish certain other objectives. (Lemma 1) These truncated Haar expansions each have $m = 2^{M+1} $ possible outcomes. Next, in Proposition 1, we show that the sum of Haar expansions having only $M$ terms is in fact the projection of a multinomial random variable. 

In Lemma 2, we identify (via the constant $b_0$ introduced there) the tails of the multinomial random variable.  After cutting off these tails, we compute the probabilities for the multinomial distribution using Stirlings's formula and Taylor series approximation (Lemma 3). The appearance of the Gaussian density on the multinomial side can be seen in this step.

On the Gaussian side, we express a standard normal random variable as a sum of $m$ independent normal random variables with coefficients being the outcomes of the truncated Haar expansion. We then apply Fubini's Theorem to 
reduce by one dimension the expression for the expected value on the Gaussian side as an integral over a hyperplane in $\mathbb{R}^{m}$ (Lemma 4).  In Lemma 5, we identify (via the constant $b_1$ introduced there) the tails of the Gaussian.  After cutting off these tails, we approximate the integral by a Riemann sum.  In Proposition 2, we pull together the results of Lemmas 4 and 5. The Riemann and the multinomial sums match perfectly. 

In Theorem 1, by bounding the function $f$ by its sup norm, we estimate the sum of the absolute values of the differences between the multinomial and Gaussian probabilities, establishing strong convergence off the tails.  It is here that we also obtain the rate of convergence of $n^{-1/2}$ and the constant for convergence of $\frac{2m^{2}}{3\sqrt{2\pi}}$, also off the tails.
In both instances, the restriction to ``off the tails'' arises since our truncations
(of the Haar expansions, the multinomial sum, and the Gaussian Riemann 
sum) are  based Chebyshev's inequality, in which coarseness is the price of its generality.
Finally, in Theorem 2, we pull together the preceding results to prove the CLT.  
\section{Preliminary Estimates}
Let $\epsilon>0$. Let $f:\mathbb{R}\rightarrow{\mathbb{R}}$ be a bounded, continuous function. Let $Z$ be a random variable on a probability space $(\Omega,\mathcal{F},P)$. We may assume that E($Z$)=0 and var($Z$)=1. Define the quantile of $Z$ to be the function $X:[0,1]\rightarrow{\mathbb{R}}$ defined by
\[
X(x) := \inf\{y\in{\mathbb{R}}|P(Z\leq{y})\geq{x}\}.
\]
Then, $X$ is a random variable on [0,1] (equipped with Lebesgue measure, on Borel sets) having the same distribution as $Z$.

For $x\in{(0,1)}$, let $\epsilon_{i}(x)$ be the ith bit in the binary expansion of $x$ (for dyadic rationals, choose the expansion with the tail of 0's). We create the following matrix of binary digits:

\[ \left( \begin{array}{cccc}
{\epsilon}_{1} & {\epsilon}_{3} & {\epsilon}_{6} \\
{\epsilon}_{2} & {\epsilon}_{5} & {\epsilon}_{9} & {\hdots} \\
{\epsilon}_{4} & {\epsilon}_{8} & {\epsilon}_{13} & {}  \\
{}&{\vdots} &{}&{}\end{array}\right)\] 
For all $x\in(0,1)$, define $P_{i}(x)$ to have binary expansion given by the ith column of the matrix. Let $X_{i}(x):=X(P_{i}(x))$. Then, $(X_{i})$ is an i.i.d. sequence of random variables on $[0,1]$ having the same distribution as $X$. 

Note that by assumption, $X\in{L^{2}{([0,1])}}$. The Haar basis is the simplest orthonormal system in $L^{2}([0,1])$ and consists of the set $S=\left\{H_{j,k}(x)|0\leq{j}<{\infty}, 0\leq{k}\leq{2^{j}-1}\}\cup\{\chi_{[0,1]}\right\}$ where

\[  H_{j,k}(x) := \left\{
\begin{array}{ll}
       2^{\frac{j}{2}} & x\in{[\frac{k}{2^{j}},\frac{k+\frac{1}{2}}{2^{j}})} \\
      \ -2^{\frac{j}{2}} & x\in{[\frac{k+\frac{1}{2}}{2^{j}},\frac{k+1}{2^{j}})} \\
         0 & $otherwise$\\
\end{array} 
\right. \]
Since $E(X)=0$ and $\|X\|<\infty$, it follows that
\[
X(x)=\sum_{j=0}^{\infty}\sum_{k=0}^{2^{j}-1}c_{j,k}H_{j,k}(x)
\]
where $c_{j,k}=\int_{0}^{1}X(x)H_{j,k}(x)\, dx.$ Then,
\[
X(x) = \sum_{j=0}^{\infty}\sum_{k=0}^{2^{j}-1}c_{j,k}2^{\frac{j}{2}}(-1)^{\epsilon_{j+1}(x)}{\chi}_{\{k\}}{(\lfloor2^{j}x\rfloor)} =\sum_{j=0}^{\infty}2^{\frac{j}{2}}c_{j,\lfloor2^{j}x\rfloor}(-1)^{\epsilon_{j+1}(x)},
\]
where, as usual, $\lfloor{x}\rfloor$ 
denotes the greatest integer $\leq{x}$.
For $n \geq 1$, define
\[
S_{n}(x) := \sum_{i=1}^{n}X_{i}(x)=\sum_{i=1}^{n}\sum_{j=0}^{\infty}2^{\frac{j}{2}}c_{j,\lfloor2^{j}P_{i}(x)\rfloor}(-1)^{\epsilon_{j+1}(P_{i}(x))},
\]
and for $M \geq 1$, define 
\begin{equation}
S_{n,M}(x) := \sum_{i=1}^{n}X_{i,M}(x)=\sum_{i=1}^{n}\sum_{j=0}^{M}2^{\frac{j}{2}}c_{j,\lfloor2^{j}P_{i}(x)\rfloor}(-1)^{\epsilon_{j+1}(P_{i}(x))}\text{ and}
\end{equation}
\begin{equation}
 \sigma_{M} := \sqrt{\sum_{j=0}^{M}\sum_{k=0}^{2^{j}-1}c_{j,k}^{2}}.
\end{equation}

\begin{lemma} 
Let $f:\mathbb{R}\rightarrow{\mathbb{R}}$ be a bounded, continuous function. Then, there exists a positive integer $M_{0}$ such that for all $M\geq{M_{0}}$:
\[
\left|\int_{0}^{1}f\left(\frac{S_{n}(x)}{\sqrt{n}}\right)\, d\lambda{(x)}-\int_{0}^{1}f\left(\frac{S_{n,M}(x)}{\sigma_{M}\sqrt{n}}\right)\, d\lambda{(x)}\right|<\epsilon(6\|f\|_{\infty}+1).
\]
\end{lemma}
\begin{proof}
Note that $E\left(\frac{S_{n}(x)}{\sqrt{n}}\right)=0$ and var$\left(\frac{S_{n}(x)}{\sqrt{n}}\right)=1$. Let 
\[
A:=\left\{\left|\frac{S_{n}}{\sqrt{n}}\right|>L\right\}\text{ and } B_{M}:=\left\{\left|\frac{S_{n,M}}{\sigma_{M}\sqrt{n}}\right|>L\right\}.
\]
By Chebyshev's inequality,
\[
\lambda(A)\leq{\frac{1}{L^{2}}}<\epsilon\text{ and }\lambda(B_{M})\leq{\frac{1}{L^{2}}}<\epsilon
\]
for $L$ large. Since $f$ is uniformly continuous on $[-L,L]$, then there exists a $\delta>0$ such that $x,y\in{[-L,L]}$ satisfying $|x-y|<\delta$ implies that $|f(x)-f(y)|<\epsilon.$ Now, let 
\[
C_{M}:=\left\{\left|\frac{S_{n}}{\sqrt{n}}-\frac{S_{n,M}}{\sigma_{M}\sqrt{n}}\right|\geq{\delta}\right\}.
\]
There exists an $M_{0}\in{\mathbb{N}}$ such that for all $M\geq{M_{0}}$:
\[
\text{var}\left(\frac{S_{n}}{\sqrt{n}}-\frac{S_{n,M}}{\sigma_{M}\sqrt{n}}\right)
\leq(1-{\sigma_{M}}^{2})+2(1-{\sigma_{M}})\sqrt{1-{\sigma_{M}^{2}}}+(1-\sigma_{M})^{2}<\epsilon\delta^{2}
\text{ and so} \]
\[
\lambda(C_{M})\leq{\frac{\epsilon\delta^{2}}{\delta^{2}}}=\epsilon.
\]
Now, let $S:=A^{c}\cap{B_{M}}^{c}\cap{C_{M}}^{c}$.
Then, $\lambda(S^{c})<3\epsilon.$ Hence,
\[
\left|E\left(f\left(\frac{S_{n}}{\sqrt{n}}\right)\right)-E\left(f\left(\frac{S_{n,M}}{\sigma_{M}\sqrt{n}}\right)\right)\right|
\leq{2\|f\|_{\infty}{\lambda}(S^{c})}+{\epsilon}{\lambda(S)} \leq{\epsilon(6\|f\|_{\infty}+1)}.
\] 
\end{proof}
Let
\[
X_{M}(x) := \frac{1}{\sigma_{M}}\sum_{j=0}^{M}2^{\frac{j}{2}}c_{j,\lfloor2^{j}x\rfloor}(-1)^{\epsilon_{j+1}(x)}
\]
for $x\in[0,1]$.  Below, we investigate the properties of this random variable. Note that $X_{M}$ is a random variable which depends on $(\epsilon_{1},...,\epsilon_{M+1})$.   From now on, we will let $m := 2^{M+1}$ for notational convenience.
Thus, $X_{M}$ is constant on dyadic intervals of length $2^{-(M+1)} \left ( = \frac{1}{m}\right )$. Let $o_{1},...,o_m$ denote the $m$ values.
It follows that 
\[
\sum_{i=1}^mo_{i}=0\text{ and }\sum_{i=1}^mo_{i}^{2}=m
\]
as $E(X_{M})=0$ and $\text{var}(X_{M})=1$.

We will now take a closer look at the random variable $S_{n,M}$ 
of Equation (1).
Each $X_{i,M}$ is a random variable with the $m$ possible outcomes $o_{1},...,o_m$.  Let $K_{i}$ be the random variable which denotes the number of times the outcome $o_{i}$ is observed among $n$ independent trials, having outcomes $k_{i}$. Then,
\begin{equation}
S_{n,M}(x)=K_{1}(x)o_{1}+...+K_m(x)o_m 
\end{equation}
where $K_{1}+...+K_m = n$. 

\newpage

Note that $S_{n,M}$ is a scalar product of an $m$-nomial random variable and the vector of outcomes.  Since each outcome has 
probability $\frac{1}{m}$ and the trials are independent, 
\[
\lambda\left (\left \{
x\in{(0,1)}:K_{1}(x)=k_{1},...,K_m(x)=k_m\right\}\right )={n\choose{k_{1},...,k_m}}\left(\frac{1}{m}\right)\text{ and}
\]
\[
 E\left(f\left(\frac{S_{n,M}(x)}{\sqrt{n}}\right)\right)=\underset{k_{1}+...+k_m = n}{\sum_{k_{1}=0}^{n}...\sum_{k_m=0}^{n}} \frac{1}{m}{n\choose{k_{1},...,k_m}}f\left( \frac{\sum_{i=1}^{m}k_{i}o_{i}}{\sqrt{n}}\right).
\]

\begin{proposition}
Let $X$ be a random variable on [0,1] having mean $0$ and variance $1$, let
\[
S_{n}(x) := \sum_{i=1}^{n}X(P_{i}(x)),
\]
and for each $M > 0$, let $S_{n,M}$ be as in Equation (1) and $\sigma_M$ be
as in Equation (2).  Then,
for each bounded and continuous $f:\mathbb{R} \to \mathbb{R}$, and
each $\epsilon>0$, there exists $M_{0}\in\mathbb{N}$ such that for all $M \geq M_0$,
\[
\left|\int_{0}^{1}f\left(\frac{S_{n}(x)}{\sqrt{n}}\right)\, d\lambda{(x)}-\int_{0}^{1}f\left(\frac{S_{n,M}(x)}{\sigma_{M}\sqrt{n}}\right)\, d\lambda{(x)}\right| < \epsilon(6\|f\|_{\infty}+1).
\]
\end{proposition}

\begin{proof}
The theorem follows from Lemma 1 and the discussion following Lemma 1.

\end{proof}

The following lemma allows us to cut off the tails from the multinomial random variable. Consequently, we prepare the ground for the usage of Taylor's formula. The tails of the multinomial random variable consist of all $(k_{1},...,k_{m})\in{\{0,1,...,n\}^{m}}$ such that $k_{1}+...+k_{m}=n$ and $k_{i}\notin[\lfloor\frac{n}{m}\rfloor-\lfloor{b\sqrt{n}}\rfloor,\lfloor\frac{n}{m}\rfloor+\lfloor{b\sqrt{n}}\rfloor]$ for all $1\leq{i}\leq{m-1}$.
\begin{lemma}
Let
\[
q(n,k_{1},...k_{m}):=\left(\frac{1}{m}\right)^{n}{n\choose{k_{1},...,k_{m}}}f\left(\frac{\sum_{i=1}^{m}k_{i}o_{i}}{\sqrt{n}}\right).
\]
Then, there exists a $b_{0}$ such that for all $b\geq{b_{0}}$:
\[
\left|\underset{k_{1}+...+k_{m}=n}{\sum_{k_{1}=1}^{n}...\sum_{k_{m}=0}^{n}}q(n,k_{1},...k_{m})-\underset{k_{1}+...+k_{m}=n}{\sum_{k_{1}=\lfloor\frac{n}{m}\rfloor-\lfloor{b\sqrt{n}}\rfloor}^{\lfloor\frac{n}{m}\rfloor+\lfloor{b\sqrt{n}}\rfloor}...\sum_{k_{m-1}=\lfloor\frac{n}{m}\rfloor-\lfloor{b\sqrt{n}}\rfloor}^{\lfloor\frac{n}{m}\rfloor+\lfloor{b\sqrt{n}}\rfloor}}{q(n,k_{1},...,k_{m}})\right|<{\epsilon}\|f\|_{\infty}.
\]
\end{lemma}
\begin{proof}
Recall that $K_{i}$ is the random variable which denotes the number of times the outcome $o_{i}$ is observed, having values $k_{i}$. As each $K_{i}$ is a binomial random variable, we have $E(K_{i})=\frac{n}{m}$ and var$(K_{i})=n(\frac{1}{m})(1-\frac{1}{m}).$ By Chebyshev's inequality,  
\[
\lambda\left( |K_{i}-\frac{n}{m}|\geq{b\sqrt{n}}\right)\leq{\frac{{(1-\frac{1}{m})}}{b^{2}m}\leq{\frac{1}{b^{2}m}}}.
\]
Then, there exists a $b_{0}$ such that for all $b\geq{b_{0}}$:
\[
\lambda\left( |K_{i}-\frac{n}{m}|\geq{b\sqrt{n}}\text{ for some }1\leq{i}\leq{m}\right)\leq\frac{1}{b^{2}}<\epsilon.
\]
\end{proof}
In the following lemma, we will use Stirling's formula and Taylor series to approximate the probabilities for the multinomial distribution.
For use here and in the proof of Theorem 2, we define some functions.  For $n > 0$, we let:
\begin{equation}
d_n := \sqrt{m}\left (\frac{m}{2n\pi}\right )^{\frac{m-1}{2}}.
\end{equation}
For $n > 0$ and integers, $j_1,\ \ldots ,\ j_m$
whose sum is 0, we let:
\begin{equation}
H(n,j_{1},...,j_{m}):=\left(\frac{m^{2}}{4n^{2}}-\frac{m}{2n}\right)\sum_{i=1}^{m}j_{i}^{2}+\left(\frac{m^{2}}{6n^{2}}-\frac{m^{3}}{6n^{3}}\right)\sum_{{i}=1}^{m}j_{i}^{3}
-\frac{m^{3}}{3n^{3}}\sum_{i=1}^{m}j_{i}^{4}+O(n^{-1}), \text{and}
\end{equation}
\begin{equation}
p(n,j_{1},...,j_{m}) :=d_ne^{H(n,j_{1},...,j_{m})}.
\end{equation}
\begin{lemma} Let $j_{i}=k_{i}-\lfloor\frac{n}{m}\rfloor,$ and suppose that $-\lfloor{b\sqrt{n}}\rfloor\leq{j_{i}}\leq{\lfloor{b\sqrt{n}}\rfloor}$ for $1\leq{i}\leq{m-1}$ and $j_{m}=-j_{1}-...-j_{m-1}$ and $n\geq{b^{2}m^{2}}$. Then,
\[
\frac{1}{m^{n}}{n\choose{k_{1},...,k_{m}}}=\left(1+O\left(\frac{1}{n}\right)\right)p(n,j_{1},...,j_{m}).
\]
\end{lemma}

\begin{proof}
Set
\[
l(n,k_{1},...,k_{m}):=\frac{1}{m^{n}}{n\choose{k_{1},...,k_{m}}}.
\]
By Stirling's Formula, we have
\[
 l(n,k_{1},...,k_{m}) =\frac{\left(1+O\left(\frac{1}{n}\right)\right)(2{\pi})^{\frac{1}{2}}n^{n+\frac{1}{2}}}{(2{\pi})^{\frac{m}{2}}{m^{n}}k_{1}^{(k_{1}+\frac{1}{2})}...k_{m}^{(k_{m}+\frac{1}{2})}}.
\]
Letting $k_{i}=\frac{n}{m}+j_{i}$ for $1\leq{i}\leq{m}$,

\begin{eqnarray*}
l(n,k_{1},...,k_{m})& =&\frac{\left(1+O\left(\frac{1}{n}\right)\right) (2{\pi})^{\frac{1}{2}}n^{n+\frac{1}{2}}}{(2{\pi})^{\frac{m}{2}}{m^{n}}(\frac{n}{m}+j_{1})^{(\frac{n}{m}+j_{1}+\frac{1}{2})}...(\frac{n}{m}+j_{m})^{(\frac{n}{m}+j_{m}+\frac{1}{2})}}\\ 
&=&\frac{\left(1+O\left(\frac{1}{n}\right)\right)m^{\frac{m}{2}}}{(2{\pi})^{\frac{m-1}{2}}{n}^{\frac{m-1}{2}}(1+\frac{mj_{1}}{n})^{(\frac{n}{m}+j_{1}+\frac{1}
{2})}...(1+\frac{mj_{m}}{n})^{(\frac{n}{m}+j_{m}+\frac{1}{2})}}.\\
\end{eqnarray*}
For all $1\leq{i}\leq{m}$, we set
\[
a(n,m,i):=(1+\frac{mj_{i}}{n})^{\frac{n}{m}+j_{i}+\frac{1}{2}}=e^{(\frac{n}{m}+j_{i}+\frac{1}{2})\ln(1+\frac{mj_{i}}{n})}.
\] 
Using a Taylor series approximation, we have the following for n large enough (as $j_{i}$ is bounded by O($\sqrt{n}$)):
\begin{eqnarray*}
 a(n,m,i)&=& \exp{\left(\left(\frac{n}{m}+j_{i}+\frac{1}{2}\right)\left(\frac{mj_{i}}{n}-\frac{m^{2}j_{i}^{2}}{2n^{2}}+\frac{m^{3}j^{3}}{3n^{3}}+O(n^{-1})\right)\right)}
\\  &=& \exp\left(j_{i}+\frac{mj_{i}^{2}}{2n}+\frac{mj_{i}}{2n}-\frac{m^{2}j_{i}^{3}}{6n^{2}}-\frac{m^{2}j_{i}^{2}}{4n^{2}}+\frac{m^{3}j_{i}^{4}}{3n^{3}}+\frac{m^{3}j_{i}^{3}}{6n^{3}}+O(n^{-1})\right).
\end{eqnarray*}
Therefore, we have 
\[
(1+\frac{mj_{1}}{n})^{\frac{n}{m}+j_{1}+\frac{1}{2}}...(1+\frac{mj_{m}}{n})^{\frac{n}{m}+j_{m}+\frac{1}{2}} =
e^{H(n,j_{1},...,j_{m})}\text{, and so }
\]
\[
l(n,j_{1},...,j_{m})=\left(1+O\left(\frac{1}{n}\right)\right)p(n,j_{1},...,j_{m})\text{, as required.}
\]
\end{proof}

Now we consider the Gaussian side.
Let $Y_{1},...,Y_{n}$ be i.i.d. standard normal random variables. Then, by the properties of i.i.d. normal random variables and by Lemma 2, 
\[
\frac{1}{\sqrt{m}}\sum_{i=1}^{m}o_{i}Y_{i}
\]
is a standard normal random variable.

\begin{lemma}
Let $S$ be the hyperplane  in 
$\mathbb{R}^{m}$
given by $\displaystyle{\sum_{i=1}^{m}y_{i}=0}$.
Then,
\[
E\left(f\left(\sum_{i=1}^{m}\frac{o_{i}Y_{i}}{\sqrt{m}}\right)\right)=E_{\mathcal{L(}Z)}\left(f\left(\sum_{i=1}^{m}\frac{o_{i}Y_{i}}{\sqrt{m}}\right)\right),
\]
where $\mathcal{L(}Z)$ is a standard Gaussian $m-1$ dimensional
measure on the hyperplane $S.$ Moreover,
\[
E_{\mathcal{L(}Z)}\left(f\left(\frac{\sum_{i=1}^{m}o_{i}Y_{i}}{\sqrt{m}}\right)\right) =
\frac{\sqrt{m}}{
\left( \sqrt{2\pi }\right) ^{{m-1}}}\underset{S}{\int...\int} f\left( \frac{1}{\sqrt{m}}
\sum_{i=1}^{m}o_{i}y_{i}\right) e^{ \left( -\frac{1}{2}
\sum_{i=1}^{m}y_{i}^{2}\right)}\, dy_{1}...dy_{m-1}.
\]
\end{lemma}

\begin{proof}

Set 
\begin{equation*}
o=\left[ 
\begin{array}{c}
o_{1} \\ 
. \\ 
. \\ 
o_{m}
\end{array}
\right] ,\text{ \ }Y=\left[ 
\begin{array}{c}
Y_{1} \\ 
. \\ 
. \\ 
Y_{m}
\end{array}
\right]  ,\text{ \ }u=\left[ 
\begin{array}{c}
1 \\ 
. \\ 
. \\ 
1
\end{array}
\right]. 
\end{equation*}
Then, the vector projection of $Y$ in direction of the vector $u$
is given by
\begin{equation*}
V := \left( \frac{1}{m}\sum_{i=1}^{m}Y_{i}\right) u=\left( \frac{1}{\sqrt{m}}
\sum_{i=1}^{m}Y_{i}\right) \frac{u}{\sqrt{m}}.
\end{equation*}
Thus, $V$ is a one dimensional standard normal on the line
through the origin and orthogonal to the hyperplane S.
Viewing $V$ and $Z:=Y-V$ as vector valued random variables, $V$ and $Z$ are independent as random variables in addition to being orthogonal as vectors. This can be verified by checking that all components of $Z$ are independent of all components of $V$. Since all components of $V$
are equal to
$\displaystyle{
\overset{\_}{Y} := \frac{1}{m}\sum_{i=1}^{m}Y_{i}}$,
and since the components of $Y-V$
are $Y_{i}-\overset{\_}{Y},$ then the independence of the Gaussian random variables
follows as
\begin{equation*}
E((Y_{i}-\overset{\_}{Y})\overset{\_}{Y})=E(Y_{i}\overset{\_}{Y})-E\left( 
\overset{\_}{Y}\right) ^{2}=\frac{1}{m}-\frac{1}{m}=0
\end{equation*}
for all $i=1,...,m.$ Note that $o^{T}Y=o^{T}\left( Y-V\right)+o^{T}V$.
Since $o^{T}u=0,$ 
\begin{equation*}
f\left( \frac{\sum_{i=1}^{m}o_{i}Y_{i}}{\sqrt{m}}\right) =f\left( \frac{1}{
\sqrt{m}}o^{T}Y\right) =f\left( \frac{1}{\sqrt{m}}o^{T}Z\right).
\end{equation*}
That is, 
$\displaystyle{
f\left( \frac{1}{\sqrt{m}}o^{T}Z\right)}$ 
does not depend on $V$.  Note that since the components of $Z$ satisfy 
$\displaystyle{
\sum_{i=1}^{m}\left( Y_{i}-\overset{\_}{Y}\right) =0,\ \mathcal{L(}Z)}$,
the law of $Z$, is a standard Gaussian $m-1$ dimensional
measure on the hyperplane $S.$ From the independence of $Z$ and $V$, and Fubini's Theorem, it follows that  
\begin{equation*}
Ef\left( \frac{1}{\sqrt{m}}o^{T}Y\right) =E_{\mathcal{L(}Z)}E_{\mathcal{L(V)}
}f\left( \frac{1}{\sqrt{m}}o^{T}Z\right) =E_{\mathcal{L(}Z)}f\left( \frac{1}{
\sqrt{m}}o^{T}Z\right).
\end{equation*}
Since the density of $Y$ is given by 
\begin{equation*}
\frac{1}{\left( \sqrt{2\pi }\right)^{m}}\exp \left (-\frac{1}{2}y^{T}y\right )=\frac{1}{
\left( \sqrt{2\pi }\right) ^{m-1}}\exp \left (-\frac{1}{2}z^{T}z\right )\frac{1}{\left( 
{\sqrt{2\pi} }^{}\right) }\exp\left( -\frac{1}{2}\left( \overset{\_}{y}\right)
^{2}\right) ,
\end{equation*}
where $y$, $z$, and $\overset{\_}{y}$ are the realizations of 
$Y$, $Z$, and $\overset{\_}{Y}$, respectively, it follows that
\begin{equation*}
E_{\mathcal{L(}Z)}f\left( \frac{1}{\sqrt{m}}o^{T}Z\right) =\frac{1}{\left( 
\sqrt{2\pi}\right) ^{{m-1}}}\int_{S}f\left( \frac{1}{\sqrt{m}}
\sum_{i=1}^{m}o_{i}y_{i}\right) \exp \left( -\frac{1}{2}
\sum_{i=1}^{m}y_{i}^{2}\right)\, dS.
\end{equation*}
That is, the expected value with respect to $\mathcal{L(}Z)$ is a surface
integral over the hyperplane $S.$ In arriving at the last equality we have
used that
$\displaystyle{
\overset{\_}{y}=\frac{1}{m}\sum_{i=1}^{m}y_{i}=0\text{ on }S}$.
By projecting $S$ onto the $y_{m}=0$ plane, we have
\begin{equation*}
E_{\mathcal{L(}Z)}f\left( \frac{1}{\sqrt{m}}o^{T}Z\right) =\frac{\sqrt{m}}{
\left( \sqrt{2\pi }\right) ^{m-1}}\int ...\int f\left( \frac{1}{\sqrt{m}}
\sum_{i=1}^{m}o_{i}y_{i}\right) e^{ \left( -\frac{1}{2}
\sum_{i=1}^{m}y_{i}^{2}\right)}\, dy_{1}...dy_{m-1}
\end{equation*}
where 
$\displaystyle{
y_{m}=-\sum_{i=1}^{m-1}y_{i}}$.
\end{proof}

\begin{lemma}
Let
\[
I:=\frac{\sqrt{m}}{\left( \sqrt{2\pi }\right) ^{m-1}}\underset{y_{1}+y_{2}+...+y_{m}=0}{\int_{}^{}...\int_{}^{}}f\left( \frac{1}{\sqrt{m}}
\sum_{i=1}^{m}o_{i}y_{i}\right) e^{ \left( -\frac{1}{2}
\sum_{i=1}^{m}y_{i}^{2}\right)}\, dy_{1}...dy_{m-1}.
\]
Then, there exist $n_{0}\in{\mathbb{N}}$ and $b_{1}>0$ such that for all $n\geq{n_{0}}$ and for all $b\geq{b_{1}}$,
\[
\left|I-\frac{m^{\frac{m}{2}}}{\left( \sqrt{2\pi }\right) ^{m-1}n^{\frac{
m-1}{2}}}\underset{j_{1}+j_{2}+...+j_{m}=0}{\sum_{j_{1}=-\lfloor{b\sqrt{n}}\rfloor}^{\lfloor{b\sqrt{n}}\rfloor}...\sum_{j_{m-1}=-\lfloor{b\sqrt{n}}\rfloor}^{\lfloor{b\sqrt{n}}\rfloor}}f\left( \frac{1}{\sqrt{n}}\sum_{i=1}^{m}o_{i}j_{i}
\right)e^{-\left( \frac{m}{2}\sum_{i=1}^{m}\frac{j_{i}^{2}}{n}\right)}\right|<\epsilon(\|f\|_{\infty}+1).
\]
\end{lemma}
\begin{proof}
Let 
\[
C:=\left\{{y}\in{\mathbb{R}^{m}}:y_{i}\in{[-b\sqrt{m},b\sqrt{m}]}\text{ for all }1\leq{i}\leq{m-1}\text{ and }y_{m}=-(y_{1}+...+y_{m-1})\right\}.
\]
We have 
\[
\frac{1}{\left( 
\sqrt{2\pi }\right) ^{m-1}}\int_{S}\exp \left( -\frac{1}{2}%
\sum_{i=1}^{m}y_{i}^{2}\right)\, dS =
\frac{\sqrt{m}}{\left( \sqrt{2\pi }\right) ^{m-1}}\underset{y_{1}+y_{2}+...+y_{m}=0}{\int_{}^{}...\int_{}^{}}\exp{ \left( -\frac{1}{2}
\sum_{i=1}^{m}y_{i}^{2}\right)}\, dy_{1}...dy_{m-1}=1.
\]
Thus, there exists a $b_{1}$ such that for all $b\geq{b_{1}}$,
\[
\frac{\sqrt{m}}{(\sqrt{2\pi})^{{m-1}}}\left(\underset{{y}\in{C^{c}}}{\int_{}^{}\cdot\cdot\cdot\int_{}^{}}e^{-\frac{1}{2}\sum_{i=1}^{m}y_{i}^{2}}\, dy_{1}...d_{y_{m-1}}\right)<\epsilon\text{, and so:}
\]
\[
\left|I-\frac{\sqrt{m}}{\left( \sqrt{2\pi }\right)^{m-1}}\underset{{y}\in{C}}{\int_{}^{}...\int_{}^{}}f\left( \frac{1}{\sqrt{m}}
\sum_{i=1}^{m}o_{i}y_{i}\right) e^{ \left( -\frac{1}{2}
\sum_{i=1}^{m}y_{i}^{2}\right)}\, dy_{1}...dy_{m-1}\right|<\epsilon\|f\|_{\infty}.
\]
Suppose that $b\geq{b_{1}}$. Let
\[
I_{b}:=\frac{\sqrt{m}}{\left( {2\pi{n} }\right) ^{\frac{m-1}{2}}}\underset{{y}\in{C}}{\int_{}^{}...\int_{}^{}}f\left( \frac{1}{\sqrt{m}}
\sum_{i=1}^{m}o_{i}y_{i}\right) e^{ \left( -\frac{1}{2}
\sum_{i=1}^{m}y_{i}^{2}\right)}\, dy_{1}...dy_{m-1}.
\]
Then,
\[
I_{b}=\lim_{n\rightarrow{\infty}}\frac{m^{\frac{m}{2}}}{\left( {2\pi{n} }\right)^{\frac{m-1}{2}}}\underset{j_{1}+j_{2}+...+j_{m}=0}{\sum_{j_{1}=-\lfloor{b\sqrt{n}}\rfloor}^{\lfloor{b\sqrt{n}}\rfloor}...\sum_{j_{m-1}=-\lfloor{b\sqrt{n}}\rfloor}^{\lfloor{b\sqrt{n}}\rfloor}}f\left( \frac{1}{\sqrt{n}}\sum_{i=1}^{m}o_{i}j_{i}
\right)e^{\left( -\frac{m}{2}\sum_{i=1}^{m}\frac{j_{i}^{2}}{n}\right)}.
\]
Hence, there exists an $n_{0}\in{\mathbb{N}}$ such that for all $n\geq{n_{0}}$,
\[ 
\left|I- \frac{m^{\frac{m}{2}}}{\left( {2\pi }{n}\right) ^{\frac{m-1}{2}}}\underset{j_{1}+j_{2}+...+j_{m}=0}{\sum_{j_{1}=-\lfloor{b\sqrt{n}}\rfloor}^{\lfloor{b\sqrt{n}}\rfloor}...\sum_{j_{m-1}=-\lfloor{b\sqrt{n}}\rfloor}^{\lfloor{b\sqrt{n}}\rfloor}}f\left( \frac{1}{\sqrt{n}}\sum_{i=1}^{m}o_{i}j_{i}
\right)e^{-\left( \frac{m}{2}\sum_{i=1}^{m}\frac{j_{i}^{2}}{n}\right)} \right|<\epsilon(\|f\|_{\infty}+1).
\]
\end{proof}
\begin{proposition}
Let $Y$ be a standard normal random variable. Then, for each $\epsilon>0$, for each bounded, continuous $f:\mathbb{R}\rightarrow\mathbb{R}$, there exists $n_{0}\in{\mathbb{N}}$, $b_{1}>0$ such that for all $n\geq{n_{0}}$ and for all $b\geq{b_{1}}$,
\[
\left|E(f(Y))- \frac{m^{\frac{m}{2}}}{\left( {2\pi{n} }\right) ^{\frac{m-1}{2}}}\underset{j_{1}+j_{2}+...+j_{m}=0}{\sum_{j_{1}=-\lfloor{b\sqrt{n}}\rfloor}^{\lfloor{b\sqrt{n}}\rfloor}...\sum_{j_{m-1}=-\lfloor{b\sqrt{n}}\rfloor}^{\lfloor{b\sqrt{n}}\rfloor}}f\left( \frac{1}{\sqrt{n}}\sum_{i=1}^{m}o_{i}j_{i}
\right)e^{-\left( \frac{m}{2}\sum_{i=1}^{m}\frac{j_{i}^{2}}{n}\right)} \right|<\epsilon(\|f\|_{\infty}+1).
\]
\end{proposition}
\begin{proof}
The statement is immediate from Lemma 4 and Lemma 5.
\end{proof}

\section{Main Results}

Theorem 2 gives our proof of the CLT.  The proof appeals to Theorem 1, which also establishes strong convergence off
the tails and gives a rate and constant of convergence there; the rate is $n^{-1/2}$ and the constant of convergence is $\frac{2m^{2}}{3\sqrt{2\pi}}$.  These only hold off the tails 
as we have truncated the Haar expansions, the multinomial sum, and the Gaussian Riemann sum. 
Theorem 1 combines the results of Propositions 1 and 2.  From now on, let $b\geq{\max\left\{b_{0},b_{1}\right\}}$,
where  the former is as in Lemma 2 and the latter is as in the proof of Lemma 5.   

\begin{theorem}
Let
\[
D_{n}:=\underset{j_{1}+...+j_{m}=0}{{\sum_{j_{m-1}=-\lfloor{b\sqrt{n}}\rfloor}^{\lfloor{b\sqrt{n}}\rfloor}...\sum_{j_{1}=-\lfloor{b\sqrt{n}}\rfloor}^{\lfloor{b\sqrt{n}}\rfloor}}}\left|\left(\frac{1}{m^{n}}{n\choose{\frac{n}{m}+j_{1},...,\frac{n}{m}+j_{m}}}-\frac{m^{\frac{m}{2}}}{(2{\pi}n)^{\frac{m-1}{2}}}e^{-\left( \frac{m}{2}\sum_{i=1}^{m}\frac{j_{i}^{2}}{n}\right)}\right)\right|.
\]
Then,
\[
D_{n}\leq{\frac{2m^{2}}{3\sqrt{2\pi{n}}}+O(n^{-1})}.
\]
\end{theorem}
\begin{proof}
By Lemma 3,
\[
D_{n} = {\frac{m^{\frac{m}{2}}}{(2{\pi}n)^{\frac{m-1}{2}}}\underset{j_{1}+...+j_{m}=0}{{\sum_{j_{m-1}=-\lfloor{b\sqrt{n}}\rfloor}^{\lfloor{b\sqrt{n}}\rfloor}...\sum_{j_{1}=-\lfloor{b\sqrt{n}}\rfloor}^{\lfloor{b\sqrt{n}}\rfloor}}}\left|e^{-\left( \frac{m}{2}\sum_{i=1}^{m}\frac{j_{i}^{2}}{n}\right)}\left(e^{G(n,j_{1},...,j_{m})+O(n^{-1})}-1\right)\right|}
\]
where $G(n,j_{1},...,j_{m})=(\frac{m^{2}}{4n^{2}})\sum_{i=1}^{m}j_{i}^{2}+(\frac{m^{2}}{6n^{2}}-\frac{m^{3}}{6n^{3}})\sum_{{i}=1}^{m}j_{i}^{3}
-\frac{m^{3}}{3n^{3}}\sum_{i=1}^{m}j_{i}^{4}.$
Then,
\[
D_{n}={\frac{m^{\frac{m}{2}}}{(2{\pi}n)^{\frac{m-1}{2}}} \underset{j_{1}+...+j_{m}=0}{{\sum_{j_{m-1}=-\lfloor{b\sqrt{n}}\rfloor}^{\lfloor{b\sqrt{n}}\rfloor}...\sum_{j_{1}=-\lfloor{b\sqrt{n}}\rfloor}^{\lfloor{b\sqrt{n}}\rfloor}}}e^{-\left( \frac{m}{2}\sum_{i=1}^{m}\frac{j_{i}^{2}}{n}\right)}|e^{G(n,j_{1},...,j_{m})+O(n^{-1})}-1|} \leq
\]
\[
\frac{m^{\frac{m}{2}}}{(2{\pi}n)^{\frac{m-1}{2}}}\underset{j_{1}+...+j_{m}=0}{\sum_{j_{m-1}=-\lfloor{b\sqrt{n}}\rfloor}^{\lfloor{b\sqrt{n}}\rfloor}...\sum_{j_{1}=-\lfloor{b\sqrt{n}}\rfloor}^{\lfloor{b\sqrt{n}}\rfloor}}e^{-\left( \frac{m}{2}\sum_{i=1}^{m}\frac{j_{i}^{2}}{n}\right)}|G(n,j_{1},...,j_{m})+O(n^{-1})|.
\]
All of the terms which decay more quickly than $\frac{1}{\sqrt{n}}$ are absorbed into the error term $O(\frac{1}{n})$. It 
remains to compute the constant $C(m)$. 
Define
\[
L_{n}:=\frac{m^{\frac{m}{2}}}{(2{\pi}n)^{\frac{m-1}{2}}}\left(\frac{m^{2}}{6n^{2}}\right)\underset{j_{1}+...+j_{m}=0}{\sum_{j_{m-1}=-\lfloor{b\sqrt{n}}\rfloor}^{\lfloor{b\sqrt{n}\rfloor}}...\sum_{j_{1}=-\lfloor{b\sqrt{n}}\rfloor}^{\lfloor{b\sqrt{n}}\rfloor}}e^{-\left( \frac{m}{2}\sum_{i=1}^{m}\frac{j_{i}^{2}}{n}\right)}\left ( \sum_{i=1}^{m-1}|j_{i}|^{3}+|j_{m}|^{3}\right ).
\]
We let
\[
E_{n}:=\frac{m^{\frac{m}{2}}}{(2{\pi}n)^{\frac{m-1}{2}}}\left(\frac{m^{2}}{6n^{2}}\right)\underset{j_{1}+...+j_{m}=0}{\sum_{j_{m-1}=-\lfloor{b\sqrt{n}}\rfloor}^{\lfloor{b\sqrt{n}}\rfloor}...\sum_{j_{1}=-\lfloor{b\sqrt{n}}\rfloor}^{\lfloor{b\sqrt{n}}\rfloor}}e^{-\left( \frac{m}{2}\sum_{i=1}^{m}\frac{j_{i}^{2}}{n}\right)}\left ( \sum_{i=1}^{m-1}|j_{i}|^{3}\right ) \leq
\]
\[
\frac{m^{\frac{m}{2}}}{(2{\pi}n)^{\frac{m-1}{2}}}\left(\frac{m^{2}}{6n^{2}}\right)\underset{j_{1}+...+j_{m}=0}{\sum_{j_{m-1}=-\lfloor{b\sqrt{n}}\rfloor}^{\lfloor{b\sqrt{n}}\rfloor}...\sum_{j_{1}=-\lfloor{b\sqrt{n}}\rfloor}^{\lfloor{b\sqrt{n}}\rfloor}}e^{-\left( \frac{m}{2}\sum_{i=1}^{m-1}\frac{j_{i}^{2}}{n}\right)}\left ( \sum_{i=1}^{m-1}|j_{i}|^{3}\right ).
\]
Approximating the sum by an integral, we have 
\[
E_{n}<(m-1)\sqrt{m}\cdot\frac{m^{2}}{6n^{2}}\cdot\left(\frac{n}{m}\right)^{3/2}E(|X|^{3})+O(n^{-1})
\]
where $X$ is a standard normal random variable.
Thus,
\[
E_{n}<\frac{2}{\sqrt{2\pi}}\frac{m(m-1)}{3\sqrt{n}}+O(n^{-1}).
\]
Now, consider
\[
F_{n}:=\left(\frac{m^{2}}{6n^{2}}\right)\frac{m^{\frac{m}{2}}}{(2{\pi}n)^{\frac{m-1}{2}}}\sum_{j_{m-1}=-\lfloor{b\sqrt{n}}\rfloor}^{\lfloor{b\sqrt{n}}\rfloor}...\sum_{j_{1}=-\lfloor{b\sqrt{n}}\rfloor}^{\lfloor{b\sqrt{n}}\rfloor}e^{-\left( \frac{m}{2}\sum_{i=1}^{m}\frac{j_{i}^{2}}{n}\right)}|j_{m}|^{3}.
\]
By maximizing $e^{-\frac{1}{2}x^{2}}|x|^{3},$
\[
F_{n}\leq{\left(\frac{m^{2}}{6n^{2}}\right)\frac{m^{\frac{m}{2}}}{(2{\pi}n)^{\frac{m-1}{2}}}\left(\frac{n}{m}\right)^{3/2}\sum_{j_{m-1}=-\lfloor{b\sqrt{n}}\rfloor}^{\lfloor{b\sqrt{n}}\rfloor}...\sum_{j_{1}=-\lfloor{b\sqrt{n}}\rfloor}^{\lfloor{b\sqrt{n}}\rfloor}e^{-\left( \frac{m}{2}\sum_{i=1}^{m-1}\frac{j_{i}^{2}}{n}\right)}e^{-3/2}3^{3/2}}.
\]
Approximating the sum by an integral, we have
\[
F_{n}<\sqrt{m}\cdot\left(\frac{m^{2}}{6n^{2}}\right)\cdot\left(\frac{n}{m}\right)^{3/2}e^{-3/2}3^{3/2}+O(n^{-1}).
\]
Thus,
\[
F_{n}<\frac{m}{6\sqrt{n}}e^{-3/2}3^{3/2}+O(n^{-1}).
\]
Hence, we have
\[
E_{n}+F_{n}<\frac{2}{\sqrt{2\pi}}\cdot\frac{m(m-1)}{3\sqrt{n}}+\frac{m}{6\sqrt{n}}e^{-3/2}3^{3/2}+O(n^{-1}) <
\frac{2m^{2}}{3\sqrt{2{\pi}n}}+O(n^{-1}).
\]
It then follows that 
\[
D_{n}\leq{\frac{2m^{2}}{3\sqrt{2\pi{n}}}+O(n^{-1})}.
\]
\end{proof}
Finally, we prove the CLT using Lemmas 1 - 5 and Theorem 1.  Recall from Lemma 2 that:
\[
q(n,k_{1},...k_{m}):=\left(\frac{1}{m}\right)^{n}{n\choose{k_{1},...,k_{m}}}f\left(\frac{\sum_{i=1}^{m}k_{i}o_{i}}{\sqrt{n}}\right).
\]
Recall from Lemma 3 that:
\[
d_n := \sqrt{m}\left (\frac{m}{2n\pi}\right )^{\frac{m-1}{2}}.
\]
For $n > 0$ and integers, $j_1,\ \ldots ,\ j_m$
whose sum is 0:
\[
H(n,j_{1},...,j_{m}):=\left(\frac{m^{2}}{4n^{2}}-\frac{m}{2n}\right)\sum_{i=1}^{m}j_{i}^{2}+\left(\frac{m^{2}}{6n^{2}}-\frac{m^{3}}{6n^{3}}\right)\sum_{{i}=1}^{m}j_{i}^{3}
-\frac{m^{3}}{3n^{3}}\sum_{i=1}^{m}j_{i}^{4}+O(n^{-1}), \text{and}
\]
\[
p(n,j_{1},...,j_{m}) :=d_ne^{H(n,j_{1},...,j_{m})}.
\]
\begin{theorem}
Let $(X_{i})$ be a sequence of i.i.d. random variables with mean $\mu$ and variance $\sigma^{2}$. Let $f:\mathbb{R}\rightarrow{\mathbb{R}}$ be a bounded, continuous function. Then, for each $\epsilon>0$, there exists $n_{1}\in\mathbb{N}$ such that 
\[
\left|E\left(f\left(\frac{X_{1}+...+X_{n}-n\mu}{\sigma\sqrt{n}}\right)\right)-E(f(Y))\right \vert<\epsilon(9\|f\|_{\infty}+2)
\]
for all $n\geq{n_{1}}$, where $Y$ is a standard normal random variable.
\end{theorem}
\begin{proof}
By Lemma 1, we reduce the problem to dealing with the projection of a multinomial random variable and we have
\[
\Delta_{n}:=\left|E\left(f\left(\frac{S_{n}}{\sqrt{n}}\right)\right)-E(f(Y))\right|<\left|E\left(f\left(\frac{S_{n,M}}{\sigma_{M}\sqrt{n}}\right)\right)-E(f(Y))\right|+{\epsilon}(6\|f\|_{\infty}+1) =
\]
\[
\left| \underset{k_{1}+...+k_{m}=n}{\sum_{k_{1}=0}^{n}...\sum_{k_{m}=0}^{n}}q(n,k_{1},...k_{m}) - E(f(Y))\right|+{\epsilon}(6\|f\|_{\infty}+1).
\]
By Lemma 2, we cut off the tails of the multinomial random variable to obtain
\[
\Delta_{n}<\left|\underset{k_{1}+...+k_{m}=n}{\sum_{k_{1}=\lfloor\frac{n}{m}\rfloor-\lfloor{b\sqrt{n}}\rfloor}^{\lfloor\frac{n}{m}\rfloor+\lfloor{b\sqrt{n}}\rfloor}...\sum_{k_{m-1}=\lfloor\frac{n}{m}\rfloor-\lfloor{b\sqrt{n}}\rfloor}^{\lfloor\frac{n}{m}\rfloor+\lfloor{b\sqrt{n}}\rfloor}}q(n,k_{1},...k_{m}) - E(f(Y))\right|+\epsilon(7\|f\|_{\infty}+1).
\]
By Lemma 3, we further simplify the multinomial sum to obtain
\[
\Delta_{n}<\left|\underset{j_{1}+...+j_{m}=0}{\sum_{j_{m-1}=-\lfloor{b\sqrt{n}}\rfloor}^{\lfloor{b\sqrt{n}}\rfloor}...\sum_{j_{1}=-\lfloor{b\sqrt{n}}\rfloor}^{\lfloor{b\sqrt{n}}\rfloor}}p(n,j_{1},...,j_{m})f\left( \frac{\sum_{i=1}^{m}j_{i}o_{i}}{\sqrt{n}}\right)-E(f(Y))\right|
+\epsilon(7\|f\|_{\infty}+1).
\]
Writing $Y$ as a sum of $m$ independent standard normal random variables, it follows that \\
$\frac{1}{\sqrt{m}}\sum_{i=1}^{m}o_{i}Y_{i}\overset{d}{=}N(0,1)$. By Lemma 4,
\[
E\left(f\left(\sum_{i=1}^{m}\frac{o_{i}Y_{i}}{\sqrt{m}}\right)\right)=E_{\mathcal{L(}Z)}\left(f\left(\sum_{i=1}^{m}\frac{o_{i}Y_{i}}{\sqrt{m}}\right)\right).
\]
\noindent By Lemma 5, we approximate the Gaussian integral by a Riemann sum. This approximation allows us to match the multinomial side and the Gaussian side and apply Theorem 1:
\[
\Delta_{n} < d_n\left\vert\underset{j_{1}+...+j_{m}=0}{\sum_{j_{m-1}=-\lfloor{b\sqrt{n}}\rfloor}^{\lfloor{b\sqrt{n}}\rfloor}...\sum_{j_{1}=-\lfloor{b\sqrt{n}}\rfloor}^{\lfloor{b\sqrt{n}}\rfloor}}
f\left( \frac{\sum_{i=1}^{m}j_{i}o_{i}}{\sqrt{n}}\right)\left ( e^{H(n,j_{1},...,j_{m})} - e^{-\left( \frac{m}{2}\sum_{i=1}^{m}\frac{j_{i}^{2}}{n}\right)}\right )\right\vert + \epsilon(8\|f\|_{\infty}+2) < 
\]
\[
\epsilon(8\|f\|_{\infty}+2)+\frac{2m^{2}\|f\|_{\infty}}{3\sqrt{2\pi{n}}}+O(n^{-1})<\epsilon(9\|f\|_{\infty}+2)
\]
for all $n\geq{n_{1}}$.
\end{proof}

\end{document}